\documentclass[preprint,3p,round,authoryear]{elsarticle}

\makeatletter
\def\ps@pprintTitle{%
   \let\@oddhead\@empty
   \let\@evenhead\@empty
   \let\@oddfoot\@empty
   \let\@evenfoot\@empty
}
\makeatother

\usepackage{amssymb}
\usepackage{amsmath}
\usepackage{amsthm} 

\newtheorem{definition}{Definition}
\newtheorem{theorem}{Theorem}

\newtheorem{lemma}{Lemma}
\usepackage{enumitem} 

\usepackage[authoryear]{natbib}  

\usepackage[colorlinks=true, allcolors=blue]{hyperref}

\journal{}

\begin{document}

\begin{frontmatter}



\title{Stationary Solution of $p$-Order Cloud Model via Stochastic Recurrence Equation}


\author[a]{Biao Hu\corref{mycorrespondingauthor}}
\cortext[mycorrespondingauthor]{School of Big Data and Internet of Things, Chongqing Vocational Institute of Engineering, Chongqing 402260, China.}
\ead{cs_hubiao@163.com}


\author[a]{Minyue Wang}

\affiliation[a]{organization={School of Big Data and Internet of Things},
            addressline={Chongqing Vocational Institute of Engineering}, 
            city={Chongqing},
            postcode={402260}, 
            country={China}}


\begin{abstract}
This paper investigates the generative mechanism of the $p$-order cloud model, which is a mathematical framework for representing uncertainty with applications in image processing, evaluation, and decision-making systems. By employing a reparameterization technique, we reformulate the cloud model as a stochastic recurrence equation (SRE) with a nonlinear transformation involving an absolute value. Under standard assumptions of stationarity, ergodicity, and an appropriate integrability condition, we establish the existence and uniqueness of a stationary solution. In particular, we demonstrate that the logarithmic moment of the model’s coefficient, modeled as a standard normal random variable, is negative, thereby ensuring almost sure convergence. These results provide new insights into the stochastic stability of cloud models and offer a rigorous foundation for further theoretical and practical developments in uncertainty quantification.
\end{abstract}




\end{frontmatter}



\section{Introduction}
\label{Introduction}

Cloud model\citep{li1995membership} is a mathematical framework for representing uncertainty. It has been extensively applied in image processing\citep{dai2023optimized}, evaluation\citep{liu2023risk,zhou2021operation}, and intelligent decision-making systems\citep{liu2021multistage,song2021goal}. Theoretical studies mainly focus on the Gaussian cloud model (hereafter referred to as the cloud model, unless stated otherwise), with key topics including parameter estimation\citep{xu2014new} and similarity measurement\citep{dai2022uncertainty, li2011similarity}. Studies on its mathematical properties, particularly the statistical characteristics of second-order cloud models, are relatively well explored. Notably, the probability density function of the random variable $X$, formed by cloud droplets $x$, lacks an analytical solution and exhibits sharp peaks and heavy tails\citep{liu2005Statistical}. Research on higher-order cloud models primarily involves moment analysis and probability density estimation\citep{liu2012Characters}. Most existing studies involve statistical analyses of predefined cloud models or conceptual discussions limited to their fundamental definitions. However, the generative mechanism of cloud models remains underexplored. This study investigates the implicit stochastic process in $p$-order cloud models.  

This paper applies the reparameterization trick\citep{rezende2014stochastic} to redefine the $p$-order cloud model\citep{wang2013p}, making its generative process explicit. We show that the $p$-order cloud model is a specific form of the stochastic recurrence equation (SRE)\citep{kesten1973random} and analyze its limiting distribution.  

\section{Preliminaries}
\label{Preliminaries}

\subsection{Cloud model}

Cloud model quantifies statistical information using numerical characteristics, analogous to a Gaussian distribution defined by its mean and variance. A $p$-order cloud model is characterized by $p+1$ numerical parameters: $En_1, En_2, \dots, En_{p-1}, En_p, He$, where $Ex = En_1$ represents the expectation, and the remaining $p$ parameters quantify the model's uncertainty. The recursive formulation of the $p$-order cloud model\citep{wang2013p} is presented in Definition \ref{def1}. We reformulate Definition \ref{def1} via the reparameterization trick, as shown in Definition \ref{def2}.  

\begin{definition} \label{def1}
A random variable $X$ is said to follow a $p$-order cloud model, parameterized by $p+1$ numerical characteristics:  
$$
En_1, En_2, \dots, En_{p-1}, En_p, He
$$
where $He > 0$. This is denoted as $X \sim \mathcal{CM}(En_1, En_2, \dots, En_{p-1}, En_p, He)$.  

As defined in Equation (\ref{E1}), the $p$-th iteration of the Gaussian random realization yields $x_p$, termed a cloud drop of the $p$-order cloud model, representing a sample from $X$:

\begin{equation} \label{E1}
x_i = \left\{
\begin{aligned}
&R_N(En_p, He),& \quad i=1  \\
&R_N(En_{p-(i-1)},|x_{i-1}|),& \quad 2\leq i \leq p
\end{aligned}
\right.
\end{equation}
where $R_N(\mu, \sigma)$ denotes a realization of a Gaussian random variable with mean $\mu$ and variance $\sigma^2$, i.e., $X \sim \mathcal{N}(\mu, \sigma^2)$.
\end{definition}

\begin{definition} \label{def2}
Let the random variable $X$ be parameterized by a $p$-order cloud model, denoted as
$$
X \sim \mathcal{CM}(En_1, En_2, \dots, En_{p-1}, En_p, He).
$$
Then, its realization follows the recurrence relation:
\begin{equation} \label{E2}
x_i = \left\{
\begin{aligned}
&En_p + He \cdot \epsilon, &\quad i=1  \\
&En_{p-(i-1)} + |x_{i-1}| \cdot \epsilon, &\quad 2\leq i \leq p
\end{aligned}
\right.
\end{equation}
where $\epsilon \sim \mathcal{N}(0,1)$ is a standard normal random variable.
\end{definition}

\subsection{Stochastic recurrence equation}

We provide a concise introduction to the stochastic recurrence equation (SRE) and its stationary distribution. The SRE is given by  
\begin{equation} \label{E3}
X_{n+1} = A_{n+1} X_{n} + B_{n+1}, \quad n \in \mathbb{Z},
\end{equation}  
where $\{(A_i, B_i), i \in \mathbb{Z} \}$ is a stationary and ergodic sequence, and $X_0$ is an initial value independent of $(A_i, B_i)$.  

When the stochastic recurrence equation meets the conditions of Theorem \ref{Theorem1}, the sequence $X_n$ converges to a random variable $X$. The limiting variable $X$ is the unique solution to  
\begin{equation} \label{E4}
X \overset{d}{=} A X + B,
\end{equation}  
where $A$ and $B$ are random variables representing a realization of $(A_n, B_n)$. Here, $\overset{d}{=}$ denotes equality in distribution. If the process \eqref{E3} is initialized with $X_0 = X$, it remains stationary, meaning that $X_n$ is invariant in distribution over time.

\begin{theorem}[\cite{brandt1986stochastic}]\label{Theorem1}
Let $\{(A_n, B_n)\}$ be a stationary and ergodic sequence. If either of the following conditions holds:
\begin{equation} \label{E5}
\mathbb{E}[\log{|A_n|}]<0 \quad \text{and} \quad \mathbb{E}[\log^{+}|B_n|]<\infty,
\end{equation}
where $\log^{+}x = \max\{\log{x},0\}$,  
then the stochastic recurrence equation (\ref{E3}) admits a unique stationary solution, given by  
\begin{equation} \label{E7}
X_n = \sum_{k=0}^{\infty} \left( \prod_{i=1}^{k} A_{n-i} \right) B_{n-k-1}.
\end{equation}
\end{theorem}

\section{Stochastic stability of $p$-order cloud model}
 In this section, we analyze the stochastic stability of the $p$-order cloud model. First, based on Definition \ref{def2}, we show that the $p$-order cloud model can be formulated as a special case of a stochastic recurrence equation (SRE). Then, within the SRE framework, we investigate its stochastic stability and prove that, under the conditions of Theorem 2, the $p$-order cloud model admits a unique stationary solution.

Within the context of stochastic recurrence equations, the sequences $\{En_{p-i}\}$ and $\{\epsilon_i\}$ correspond to $\{B_i\}$ and $\{A_i\}$, respectively, for $i \in \mathbb{Z}$, with the initial condition $X_0 = He$. Consequently, the $p$-order cloud model can be interpreted as a stochastic recurrence equation with a nonlinear transformation involving the absolute value, as given by the following equation:  
\begin{equation} \label{E8}
X_{n+1} = A_{n+1} |X_{n}| + B_{n+1}, \quad n \in \mathbb{Z},
\end{equation}
where $A \sim \mathcal{N}(0, I)$.

\begin{theorem}\label{Theorem2}
    If the sequence $\{(A_n, B_n)\}$ is stationary and ergodic, and the following condition is satisfied:  
\begin{equation*}  
\mathbb{E}[\log{|A_n|}] < 0 \quad \text{and} \quad \mathbb{E}[\log^{+}|B_n|] < \infty,  
\end{equation*}  
then Equation (\ref{E8}) has a unique stationary solution.
\end{theorem}

\begin{proof}
For each integer $k \ge 1$, define a partial solution $\{X_n^{(-k)}\}_{n \ge -k}$ by setting
$$
X_{-k}^{(-k)}(\omega) = 0,
$$

$$
X_{n+1}^{(-k)}(\omega) = A_n(\omega)\,|X_n^{(-k)}(\omega)| + B_n(\omega), \quad n \ge -k,
$$
where $\omega$ is an element of the probability space $(\Omega, \mathcal{F}, P)$. For each fixed $n \in \mathbb{Z}$ (with $-k \le n$), consider the sequence $\{X_n^{(-k)}(\omega)\}$. We show that this sequence converges almost surely as $k \to \infty$ and define
$$
X_n(\omega) = \lim_{k \to \infty} X_n^{(-k)}(\omega).
$$
Define the auxiliary sequence
$$
\tilde{X}_n^{(-k)} := |X_n^{(-k)}|.
$$
Taking absolute values in the recursion, we obtain
$$
|X_{n+1}^{(-k)}| \le |A_n|\,|X_n^{(-k)}| + |B_n|,
$$
or equivalently,
$$
\tilde{X}_{n+1}^{(-k)} \le |A_n|\,\tilde{X}_n^{(-k)} + |B_n|.
$$
Now, consider the stochastic recurrence equation
$$
Y_{n+1} = |A_n|\,Y_n + |B_n|, \quad Y_{-k} = 0.
$$
By standard results \citep{brandt1986stochastic} and the ergodic theorem, the product $\prod_{i=j+1}^{n-1} |A_i|$ decays exponentially almost surely due to $\mathbb{E}[\log |A_n|] < 0$. Consequently, the series
$$
\sum_{j=-k}^{n-1} \left( \prod_{i=j+1}^{n-1} |A_i| \right) |B_j|
$$
converges absolutely almost surely. Thus, the solution $Y_n$ is bounded. Since $\tilde{X}_n^{(-k)} \le Y_n$, it follows that $\tilde{X}_n^{(-k)}$ is uniformly bounded with respect to $k$. By a standard Cauchy sequence argument, for each fixed $n$, the sequence $\{X_n^{(-k)}\}$ is Cauchy almost surely. Therefore, the limit $X_n(\omega) = \displaystyle\lim_{k \to \infty} X_n^{(-k)}(\omega)$ exists for almost all $\omega$.

Assume that there exist two stationary solutions, $\{X_n\}$ and $\{X'_n\}$, satisfying
$$
X_{n+1} = A_n\,|X_n| + B_n \quad \text{and} \quad X'_{n+1} = A_n\,|X'_n| + B_n, \quad \forall\, n \in \mathbb{Z}.
$$
Define the difference $\Delta_n = X_n - X'_n$. Subtracting the two recursions yields $\Delta_{n+1} = A_n\left(|X_n| - |X'_n|\right)$. Taking absolute values and using the inequality
$$
\bigl||X_n| - |X'_n|\bigr| \le |X_n - X'_n| = |\Delta_n|,
$$
we obtain $|\Delta_{n+1}| \le |A_n|\,|\Delta_n|$.
By iterating this inequality, we get
$$
|\Delta_n| \le \left(\prod_{i=0}^{n-1} |A_i|\right) |\Delta_0|.
$$
Since $\mathbb{E}[\log |A_n|] < 0$, the ergodic theorem \citep{vervaat1979stochastic} implies that
$$
\lim_{n \to \infty} \prod_{i=0}^{n-1} |A_i| = 0 \quad \text{a.s.}
$$
Hence, $|\Delta_n| \to 0$ almost surely. As both solutions are stationary (i.e., their finite-dimensional distributions are time-invariant), it follows that
$$
X_n = X'_n, \quad \forall\, n \in \mathbb{Z}, \quad \text{a.s.}
$$
This establishes the uniqueness of the stationary solution.

\end{proof}

\begin{lemma}\label{lem:stationary_ergodic}
Let $\{\epsilon_i, i \in \mathbb{Z}\}$ be a sequence of independent and identically distributed random variables, and let $\{E_{n_i}, i \in \mathbb{Z}\}$ be a stationary and ergodic sequence. Assume that $\{E_{n_i}, i \in \mathbb{Z}\}$ is independent of $\{\epsilon_i, i \in \mathbb{Z}\}$. Then, the sequence $\{(\epsilon_i, E_{n_i}), i \in \mathbb{Z}\}$ is stationary and ergodic.
\end{lemma}

\begin{proof}
The proof follows directly from the properties of independent and stationary ergodic processes. For a detailed discussion, refer to \citep{gray2009probability}.
\end{proof}

\begin{theorem}[Existence and Uniqueness of Stationary Solution] \label{thm:stationary_solution}
Suppose the following assumptions are satisfied:
\begin{enumerate}[label=(\textbf{A\arabic*}), leftmargin=*, align=left]
    \item \textbf{Stationarity and Ergodicity of $\{En_i\}$}:  
    The sequence $\{En_i, i \in \mathbb{Z}\}$ is assumed to be stationary and ergodic.

    \item \textbf{Independence between $\{En_i\}$ and $\{\epsilon_i\}$}:  
    The sequence $\{En_i\}$ is independent of the i.i.d. sequence $\{\epsilon_i, i \in \mathbb{Z}\}$, where $\{\epsilon_i\}$ is defined in Definition~\ref{def2}.

    \item \textbf{Integrability Condition}:  
    The logarithmic moment condition holds, i.e., $\mathbb{E}\left[\log^{+}|En_i|\right] < \infty$.
\end{enumerate}
Under these conditions, the cloud model admits a unique stationary solution.
\end{theorem}

\begin{proof}
By Lemma \ref{lem:stationary_ergodic}, the sequence $\{(\epsilon_i, En_i), i \in \mathbb{Z}\}$ is stationary and ergodic. Here, I use $A$ to represent $\epsilon$, where $A \sim \mathcal{N}(0, I)$ has the probability density function (pdf):
$$
\phi(a) = \frac{1}{\sqrt{2\pi}} \exp\left(-\frac{a^2}{2}\right).
$$
The expectation $\mathbb{E}[\log|A|]$ is given by:
$$
\mathbb{E}[\log|A|] = \int_{-\infty}^\infty \log|a| \cdot \phi(a) \, \mathrm{d}a = 2 \int_0^\infty \log a \cdot \frac{1}{\sqrt{2\pi}} \exp\left(-\frac{a^2}{2}\right) \, \mathrm{d}a.
$$
Substituting $u = a^2/2$, which implies $a = \sqrt{2u}$ and $\mathrm{d}a = \frac{1}{\sqrt{2u}} \mathrm{d}u$, the integral becomes:
$$
\mathbb{E}[\log|A|] = \frac{2}{\sqrt{2\pi}} \int_0^\infty \log\left(\sqrt{2u}\right) \frac{1}{\sqrt{2u}} \exp(-u) \, \mathrm{d}u.
$$
Simplifying $\log\left(\sqrt{2u}\right) = \frac{1}{2}(\log 2 + \log u)$, we obtain:
$$
\mathbb{E}[\log|A|] = \frac{1}{\sqrt{\pi}} \left[ \frac{\log 2}{2} \int_0^\infty \frac{1}{\sqrt{2u}} \exp(-u) \, \mathrm{d}u + \frac{1}{2} \int_0^\infty  \frac{\log u}{\sqrt{2u}} \exp(-u) \, \mathrm{d}u \right].
$$
Using the properties of the Gamma function $\Gamma(s) = \int_0^\infty u^{s-1} \exp(-u) \, \mathrm{d}u$ and its derivative $\Gamma'(s)$, we know:
\begin{itemize}
    \item $\int_0^\infty u^{-1/2} \exp(-u) \, \mathrm{d}u = \Gamma\left(\frac{1}{2}\right) = \sqrt{\pi}$.
    \item $\int_0^\infty \log u \cdot u^{-1/2} \exp(-u) \, \mathrm{d}u = \Gamma'\left(\frac{1}{2}\right) = \Gamma\left(\frac{1}{2}\right) \psi\left(\frac{1}{2}\right)$, where $\psi(s) = \frac{\mathrm{d}}{\mathrm{d}s} \log \Gamma(s)$ is the digamma function.
\end{itemize}
The known values are $\Gamma\left(\frac{1}{2}\right) = \sqrt{\pi}$ and $\psi\left(\frac{1}{2}\right) = -\gamma - 2\log 2$, where $\gamma \approx 0.5772$ is the Euler-Mascheroni constant. Substituting these values, we obtain:
$$
\mathbb{E}[\log|A|] = \frac{1}{\sqrt{\pi}} \left[ \frac{\log 2}{2} \sqrt{\pi} + \frac{\sqrt{\pi}}{2} (-\gamma - 2\log 2) \right] = \frac{-\gamma - \log 2}{2} < 0.
$$
Therefore, the conditions in Theorem \ref{Theorem2} are satisfied, completing the proof.
\end{proof}




\bibliographystyle{elsarticle-num-names}
\bibliography{references}

\end{document}